\documentclass{amsart}
\usepackage{graphicx,color,epsf}
\usepackage [cmtip,arrow]{xy}
\xyoption{all}

\usepackage{subfig}
\usepackage{float}

\usepackage {pb-diagram,pb-xy}

\usepackage{amsmath,amssymb}
\newtheorem{theorem}{Theorem}[section]
\newtheorem{lemma}[theorem]{Lemma}
\newtheorem{corollary}[theorem]{Corollary}
\newtheorem{proposition}[theorem]{Proposition}

\theoremstyle{definition}
\newtheorem{definition}[theorem]{Definition}
\newtheorem{example}[theorem]{Example}

\newtheorem{notation}[theorem]{Notation}
\newtheorem{property}[theorem]{Property}

\theoremstyle{remark}
\newtheorem{remark}[theorem]{Remark}

\definecolor{DarkBlue}{rgb}{0,0.1,0.55}

\numberwithin{equation}{section}

\newcommand {\hide}[1]{}

\newcommand {\junk}[1]{}

\newcommand {\R} {{\textnormal{R}}}
\newcommand {\s}        {\mbox{\rm sign}}

\newcommand {\C}     {{\textnormal{C}}}

\newcommand {\Real}[1]   {\mbox{$\mathbb{R}^{#1}$}}
 \newcommand {\re}         {\Real{}}

\newcommand {\Z}  {\mathbb{Z}}

\newcommand {\ZZ} {{\rm Zer}}
\newcommand {\RR} {{\mathcal R}}

\newcommand {\la}   {{\langle}}
\newcommand {\ra}   {{\rangle}}
\newcommand {\eps} {{\varepsilon}}

\newcommand {\PP}     {\mathbb{P}} 

\newcommand {\T}      {{\mbox{\rm T}}}

\newcommand {\Zer} {{\rm Zer}}

\newcommand {\Cr} {{\rm Cr}}
\newcommand {\Pos} {{\rm Pos}}
\newcommand {\Def} {{\rm Def}}
\newcommand {\card} {{\rm card}}

\def\addots{\mathinner{\mkern1mu
\raise1pt\vbox{\kern7pt\hbox{.}}
\mkern2mu\raise4pt\hbox{.}\mkern2mu
\raise7pt\hbox{.}\mkern1mu}}

\newcommand{\HH}  {\mbox{\rm H}}

\newcommand{\case}[1]{\textbf{Case #1:}}
\newcommand{\zz}{\begin{flushright}$\Box$\end{flushright}}
\newcommand{\blist}
    {\begin{list}
    {\textbf{\arabic{enumi}.}}
    {\setlength{\leftmargin}{0.365 in}
    \setlength{\rightmargin}{0.365 in}
    \usecounter{enumi}}}
\newcommand{\alist}
    {\begin{list}
    {\textbf{(\alph{enumii})}}
    {\setlength{\leftmargin}{0.365 in}
    \setlength{\rightmargin}{0.365 in}
    \usecounter{enumii}}}
\newcommand{\rlist}
    {\begin{list}
    {\textbf{(\roman{enumiii})}}
    {\setlength{\leftmargin}{0.365 in}
    \setlength{\rightmargin}{0.365 in}
    \usecounter{enumiii}}}
\newcommand{\Rlist}
    {\begin{list}
    {\textbf{(\Roman{enumii})}}
    {\setlength{\leftmargin}{0.365 in}
    \setlength{\rightmargin}{0.365 in}
    \usecounter{enumii}}}
\newcommand{\elist}{\end{list}}

\newcommand{\Ext}{\text{Ext}}
\newcommand{\ep}{\varepsilon}

\begin{document}
\title[] {Refined bounds on the number of connected components of
sign conditions on a variety}
\author{Sal Barone}
\address{Department of Mathematics, Purdue University, West Lafayette,
  IN 47906, U.S.A.}
\email{sbarone@math.purdue.edu} \author{Saugata Basu}
\address{Department of Mathematics, Purdue University, West Lafayette,
  IN 47906, U.S.A.}
\email{sbasu@math.purdue.edu}


\date{\textbf{\today}}

\maketitle
\begin{abstract}
Let $\R$ be a real closed field, $\mathcal{P},\mathcal{Q} \subset
\R[X_1,\ldots,X_k]$ finite subsets of polynomials,
with the degrees
of the polynomials in $\mathcal{P}$ (resp. $\mathcal{Q}$) bounded by
$d$ (resp. $d_0$).
Let $V \subset \R^k$ be the real algebraic variety
defined by the polynomials in $\mathcal{Q}$ and suppose that the real
dimension of $V$ is bounded by $k'$. We prove that the number of
semi-algebraically connected components of the realizations of all
realizable sign conditions of the family $\mathcal{P}$ on $V$ is bounded
by
$$
\displaylines{
\sum_{j=0}^{k'}4^j{s +1\choose j}F_{d,d_0,k,k'}(j),}$$
where $s = \card \; \mathcal{P}$,
and $$F_{d,d_0,k,k'}(j)=
\textstyle\binom{k+1}{k-k'+j+1} \;(2d_0)^{k-k'}d^j\; \max\{2d_0,d \}^{k'-j}
+2(k-j+1)
.$$

In case $2 d_0 \leq d$, the above bound can be written simply as
$$
\displaylines{
\sum_{j = 0}^{k'} {s+1 \choose j}d^{k'} d_0^{k-k'}  O(1)^{k}
= (sd)^{k'} d_0^{k-k'} O(1)^k
}
$$
(in this form the bound was suggested by J. Matousek \cite{Matousek_private}).
Our result improves in certain cases (when $d_0 \ll d$)
the best known bound of
$$
\sum_{1 \leq j \leq k'}
     \binom{s}{j} 4^{j}  d(2d-1)^{k-1}
$$
on the same number proved in
\cite{BPR8}
in the case  $d=d_0$.

The distinction between the bound $d_0$ on the degrees
of the polynomials defining the variety $V$ and the bound $d$ on the degrees
of the polynomials in $\mathcal{P}$ that appears in the new bound
is motivated by several applications in discrete geometry
\cite{Guth-Katz,Matousek11b,Solymosi-Tao,Zahl}.
\end{abstract}

\section{Introduction}
Let $\R$ be a real closed field.
We denote by $\C$ the algebraic closure of
$\R$.
Let ${\mathcal P}$ be a finite subset of $\R[X_1,\ldots,X_k]$.
A  {\em sign condition} on
${\mathcal P}$ is an element of $\{0,1,- 1\}^{\mathcal P}$.

The {\em realization} of the sign condition $\sigma$ in a semi-algebraic
set $V \subset \R^k$ is the semi-algebraic set

\begin{equation}
\label{eqn:R(Z)}
\RR(\sigma,V)= \{x\in V\;\mid\;
\bigwedge_{P\in{\mathcal P}} \s({P}(x))=\sigma(P) \}.
\end{equation}

More generally, given any first order formula $\Phi(X_1,\ldots,X_k)$,
the realization of $\Phi$
in a semi-algebraic
set $V \subset \R^k$ is the semi-algebraic set

\begin{equation}
\label{eqn:R(Z)}
\RR(\Phi,V)= \{x\in V\;\mid\;
\Phi(x)\}.
\end{equation}

We  denote
the set of zeros
of ${\mathcal P}$ in $\R^k$  (resp. in $\C^k$) by
$$
\displaylines{
\ZZ({\mathcal P},\R^k)=\{x\in \R^k\;\mid\;\bigwedge_{P\in{\mathcal P}}P(x)= 0\}
}
$$

$$
\displaylines{
(\mbox{resp. }
\ZZ({\mathcal P},\C^k)=\{x\in \C^k\;\mid\;\bigwedge_{P\in{\mathcal P}}P(x)= 0\}
).
}
$$

The main problem considered in this paper is to obtain a tight bound
on the number of semi-algebraically connected components 
of the realizations of all realizable
sign conditions of a family of polynomials
$\mathcal{P} \subset \R[X_1,\ldots,X_k]$
in a variety $\ZZ(\mathcal{Q},\R^k)$ having
dimension $k' \leq k$, in terms of $s = \card \;\mathcal{P}, k,k'$ and
the degrees of the polynomials in $\mathcal{P}$ and $\mathcal{Q}$.

\subsection{History and prior results}
The problem of bounding the number of semi-algebraically
connected components (as well
as the higher Betti numbers) has a long history. The initial results
were obtained by  Ole{\u\i}nik and Petrovski{\u\i} \cite{OP}, and later by Thom
\cite{T} and Milnor \cite{Milnor2}, who proved a bound of $O(d)^k$
on the sum of the Betti numbers of any real algebraic variety in $\R^k$
defined by polynomials of  degree at most $d$. This result has been
generalized to arbitrary semi-algebraic sets in several different
ways. The reader is referred to \cite{BPR10} for a survey
of results in this direction and the references therein.

In \cite{PR} Pollack and Roy proved a bound of ${s \choose k}O(d)^k$
on the number of semi-algebraically connected components of the realizations of all realizable
sign conditions of a family of $s$ polynomials of degrees bounded by $d$.
The proof was based on  Ole{\u\i}nik-Petrovski{\u\i}-Thom-Milnor
bounds for algebraic sets, as well as some deformation
techniques and general position arguments. Similar results due to Alon
\cite{Alon} and
Warren \cite{Warren} only on the number of
realizable sign conditions were known before.

It was soon realized that in some applications,
notably in geometric transversal theory, as well as
in bounding the complexity of the configuration space in robotics, it is
useful to study the realizations of sign conditions of a family of
$s$ polynomials in $\R[X_1,\ldots,X_k]$ restricted to a real variety
$\ZZ(Q,\R^k)$ where the real dimension of the variety $\ZZ(Q,\R^k)$
can be much smaller than $k$. In \cite{BPR95} it was shown that the number
of semi-algebraically
connected components of the realizations of all realizable sign condition of
a family, $\mathcal{P} \subset \R[X_1,\ldots,X_k]$ of $s$ polynomials,
restricted to a real variety of dimension $k'$, where the degrees of the
polynomials in $\mathcal{P} \cup \{Q\}$ are
all bounded by $d$, is bounded by ${s \choose k'}O(d)^k$.
This last result was made more precise in \cite{BPR8}
where the authors bound (for each $i$) the sum of the
$i$-th Betti number over all
realizations of realizable sign conditions of a family of polynomials
restricted to a variety of dimension $k'$ by

$$
\displaylines{
\sum_{0 \leq j \leq k' - i}
     \binom{s}{j} 4^{j}  d(2d-1)^{k-1}.
}
$$
Notice that there is no undetermined constant in the above bound, and that
it generalizes
the bound in \cite{BPR95} which is the special case with $i=0$.
The technique of the proof uses a generalization of
the Mayer-Vietoris exact sequence in conjunction with the
 Ole{\u\i}nik-Petrovski{\u\i}-Thom-Milnor bounds on the Betti numbers of
real varieties.

In a slightly different direction,
an asymptotically tight bound 
(asymptotic with respect to the number of polynomials with fixed degrees and number of variables)
$$
\sum_{\sigma\in \{0,1,-1\}^\mathcal{P}} b_0(\sigma,\R^k)\leq \frac{(2d)^k}{k!}s^k+O(s^{k-1})
$$
was proved in \cite{BPR-tight}.  
This bound has the best possible leading term (as a function of $s$) but is 
not optimal with regards to the dependence on $d$, and thus is not directly
comparable to the results in the current paper.
\subsection{Main result}

In this paper we prove a bound on the number of
semi-algebraically
connected components over all realizable sign conditions of a family
of polynomials
in a variety.
However, unlike in the above bound
the role of the degrees of the polynomials defining the variety $V$
is distinguished from the role of the degrees of the polynomials in
the family $\mathcal{P}$.
This added flexibility seems to be necessary in certain applications
of these bounds in combinatorial geometry (notably in the recent paper
by Solymosi and Tao \cite{Solymosi-Tao}). We give another application
in the theory of geometric permutations in Section \ref{sec:applications}.

Our main result is the following theorem.
\begin{theorem}
\label{thm:main}
Let $\R$ be a real closed field, and
let $\mathcal{Q},\mathcal{P} \subset \R[X_1,\dots,X_k]$ be
finite subsets of polynomials
such that $\deg(Q)\leq d_0$ for all $Q\in \mathcal{Q}$,
$\deg P =d_P$ for all $P\in \mathcal{P}$,
and the real dimension of $\Zer(\mathcal{Q},\R^k)$ is
$k' \leq k$. Suppose also that $\card\; \mathcal{P} = s$,
and for $\mathcal{I}\subseteq \mathcal{P}$
let ${d}_\mathcal{I}=\prod_{P\in \mathcal{I}}d_P$.
Then,
$$
\displaylines{
\sum_{\sigma \in \{0,1,-1\}^{\mathcal{P}}}
b_0(\RR(\sigma,\ZZ(\mathcal{Q},\R^k)))
}
$$
is at most
$$
\displaylines{
\sum_{{\mathcal{I}\subset \mathcal{P}}\atop {\#\mathcal{I}\leq k'}}4^{\#\mathcal{I}}\left(
\textstyle\binom{k+1}{k-k'+\#\mathcal{I}+1} \;(2d_0)^{k-k'}{d}_\mathcal{I}\;\max_{P\in \mathcal{I}}\{2d_0,d_P \}^{k'-\#\mathcal{I}}+2(k-\#\mathcal{I}+1)\right).
}
$$
In particular, if $d_P\leq d$ for all $P\in \mathcal{P}$, 
we have that
$$
\displaylines{
\sum_{\sigma \in \{0,1,-1\}^{\mathcal{P}}}
b_0(\RR(\sigma,\ZZ(\mathcal{Q},\R^k)))
}
$$
is at most
$$
\displaylines{
\sum_{j=0}^{k'}4^j{s +1\choose j}\left(
\textstyle\binom{k+1}{k-k'+j+1} \;(2d_0)^{k-k'}d^j\;\max\{2d_0,d \}^{k'-j}
+2(k-j+1)\right).
}
$$
\end{theorem}

\subsection{A few remarks}
\begin{remark}
The bound in Theorem \ref{thm:main} is tight (up to a factor of
$ O(1)^k$).
It is instructive to examine the two extreme cases, when $k'=0$ and $k' = k-1$
respectively. When, $k'=0$, the variety $\ZZ(\mathcal{Q},\R^k)$ is zero
dimensional,
and is a union of at most $O(d_0)^k$ isolated points. The
bound in Theorem \ref{thm:main} reduces to $O(d_0)^{k}$ in this case,
and is thus tight.

When $k' = k-1$ and $2 d_0 \leq d$,
the bound in Theorem \ref{thm:main}
is equal to
$$
\sum_{j=0}^{k-1}4^j{s +1 \choose j}\left(
\binom{k+1}{j+2}\; 2d_0d^{k-1}+2(k-j+1)\right) = d_0 O(s d)^{k-1}.
$$
The following example shows that this is the best possible (again up to
$O(1)^k$).

\begin{example}
Let $\mathcal{P}$ be the set of $s$ polynomials in $X_1,\ldots,X_{k}$ each
of which is a product of $d$ generic linear forms. Let
$\mathcal{Q} = \{Q\}$, where
$$
Q = \prod_{1 \leq i \leq d_0} (X_{k} - i).
$$
It is easy to see that in this case
the number of semi-algebraically
connected components of all realizable
strict
sign conditions
of $\mathcal{P}$
(i.e.\
sign conditions  belonging to $\{-1,+1\}^{\mathcal{P}}$)
on $\ZZ(\mathcal{Q},\R^k)$ is
equal to
$$
\displaylines{
d_0 \sum_{i = 0}^{k-1} {sd \choose i} =
d_0 (\Omega(sd))^{k-1},
}
$$
since the intersection of
$
\displaystyle{
\bigcup_{P \in \mathcal{P}}\ZZ(P, \R^{k})
}
$
with the hyperplane defined by $X_{k} = i$ for each $i$, $i= 1,\ldots,d_0$, is
homeomorphic to an union of
of $s d$ generic hyperplanes in $\R^{k-1}$, and the number of
connected components of the complement of
the union of $s d$ generic hyperplanes in $\R^{k-1}$ is precisely
$\sum_{i = 0}^{k-1} {sd \choose i}$.
\end{example}
\end{remark}

\begin{remark}
Most bounds on the number of semi-algebraically connected components
of real algebraic varieties are stated in terms of the maximum of
the degrees of the polynomials defining the variety (rather than in terms
of the degree sequence). One reason behind this is the well-known fact that a
``Bezout type'' theorem is not true for real algebraic varieties. The
number of semi-algebraically connected components (indeed even isolated
zeros) of a set of polynomials $\{P_1,\ldots,P_m \} \subset
\R[X_1,\ldots,X_k,X_{k+1}]$
with degrees $d_1,\ldots,d_m$ can
be greater than the product $d_1\cdots d_m$, as can be seen in the
following example.

\begin{example}
\label{eg:counterexample}
Let $\mathcal{P} =
\{P_1,\ldots,P_m\}\subset \R[X_1,\ldots,X_{k+1}]$ be defined as follows.
$$
\displaylines{
P_1 = \sum_{i=1}^{k} \prod_{j=1}^d (X_i - j)^2 \cr
P_j = \prod_{i=1}^{m- j+2}(X_{k+1} - i), 2\leq j \leq m.
}
$$
Let $\mathcal{P}_i = \{P_1,\ldots,P_i\}$.
Notice that for each $i, 1\leq i < m$, $\ZZ(\mathcal{P}_i,\R^{k+1})$
strictly contains $\ZZ(\mathcal{P}_{i+1},\R^{k+1})$.
Moreover, $b_0(\ZZ(\mathcal{P},\R^{k+1})) = 2d^k$, while the product of the
degrees of the polynomials in $\mathcal{P}$ is $2 d m!$. Clearly, 
for $d$ large enough 
$2d^k > 2 d m!$.
\end{example}
\end{remark}

\begin{remark}
Most of the previously known bounds on the Betti numbers of
realizations of sign conditions relied ultimately on the
 Ole{\u\i}nik-Petrovski{\u\i}-Thom-Milnor bounds on the Betti numbers of real
varieties. Since in the proofs of these bounds
the finite family of polynomials defining a given real variety is replaced
by a single polynomial by taking a sum of squares, it is not possible
to separate
out the different roles played by the degrees of the polynomials
in $\mathcal{P}$ and those in $\mathcal{Q}$. The technique used in this
paper avoids using the  Ole{\u\i}nik-Petrovski{\u\i}-Thom-Milnor bounds, but uses directly
classically known formulas for the Betti numbers of smooth, complete
intersections in complex projective space. The bounds obtained from these
formulas depend more delicately on the individual degrees of the polynomials
involved (see Corollary
\ref{cor:bijkm}),
and this allows us to separate the roles of $d$ and $d_0$ in our
proof.
\end{remark}

\subsection{Outline of the proof of Theorem \ref{thm:main}}
The main idea behind our improved bound is to reduce the problem of bounding
the number of semi-algebraically connected components of all sign conditions
on a variety to the problem of bounding the sum of the
$\mathbb{Z}_2$-Betti
numbers of certain smooth complete intersections in complex projective
space.
This is done as follows. First assume that
$\ZZ(\mathcal{Q},\R^k)$ is bounded. The general case is reduced to
this case by an initial step using the conical triviality of semi-algebraic
sets at infinity.

Assuming that $\ZZ(\mathcal{Q},\R^k)$ is bounded, and
letting $Q = \sum_{F \in \mathcal{Q}} F^2$,
we consider another polynomial $\Def(Q,H,\zeta)$ which is an
infinitesimal perturbation of $Q$.
The basic semi-algebraic set, $T$, defined by $\Def(Q,H,\zeta) \leq 0$
is a semi-algebraic subset of $\R\langle\zeta\rangle^k$ (where
$\R\langle\zeta\rangle$ is the field of algebraic Puiseux series
with coefficients in $\R$, see Section \ref{subsec:puiseux} below for properties of
the field of Puiseux series that we need in this paper). The
semi-algebraic set $T$ has the property
that for each semi-algebraically
connected component $C$ of $\ZZ(Q,\R^k)$ there exists a semi-algebraically
connected component $D$ of $T$, which is bounded over $\R$ and such that
$\lim_\zeta D = C$ (see Section \ref{subsec:puiseux}
for definition of $\lim_\zeta$).
The semi-algebraic set $T$ should be thought of as an
infinitesimal ``tube'' around $\ZZ(Q,\R^k)$, which is bounded by a smooth
hypersurface (namely, $\ZZ(\Def(Q,H,\zeta),\R\langle\zeta\rangle^k)$).
We then show it is possible to cut out a $k'$-dimensional
subvariety, $W$ in $\ZZ(\Def(Q,H,\zeta),\R\langle\zeta\rangle^k)$,
such that (for generic choice of co-ordinates) in fact
$\lim_\zeta W = \ZZ(Q,\R^k)$ (Proposition
\ref{prop:criticallocusontube}),
and moreover the homogenizations of
the polynomials defining $W$
define a non-singular complete intersection in $\PP_{\C\langle\zeta\rangle}^k$
(Proposition \ref{prop:general_pos_for_cr}).
$W$ is defined by
$k-k'$ forms of
degree at most $2 d_0$.
In order to bound the number of semi-algebraically connected components
of realizations of sign conditions of the family $\mathcal{P}$ on
$\ZZ(Q,\R^k)$,
we need to bound the number of semi-algebraically
connected components of the intersection
of $W$ with the zeros of certain infinitesimal perturbations of
polynomials in $\mathcal{P}$ (see Proposition \ref{prop:main} below).
The number of cases that we need to consider
is bounded by ${O(s) \choose k'}$, and
again each such set of polynomials define
a non-singular
complete intersection of 
$p,\; k-k' \leq p \leq k$
hypersurfaces in $k$-dimensional
projective space  over an appropriate algebraically closed field,
$k-k'$ of which are defined
by forms having degree at most $2 d_0$ and the remaining 
$m = p - k +k'$
of degree bounded by $d$.
In this situation, there are classical formulas known for the Betti numbers
of such varieties, and they imply a bound of
$\binom{k+1}{m+1}(2d_0)^{k-k'} d^{k'}+O(k)$ 
on the sum of the Betti numbers of such
varieties (see Corollary \ref{cor:bijkm} below).
The bounds on the sum of the Betti numbers of these projective complete
intersections in the algebraic closure imply
using the well-known Smith inequality (see
Theorem \ref{thm:smith}) a bound on the number of semi-algebraically
connected components of the real parts of these varieties, and in particular
the number of bounded components.
The product of the two bounds,
namely the combinatorial bound on the number of different
cases and the algebraic part depending on the degrees, summed
appropriately lead to the claimed bound.

\subsection{Connection to prior work}
The idea of approximating an arbitrary real variety of
dimension $k'$ by a complete intersection was used in \cite{BPR95b} to give
an efficient algorithm for computing sample points in each semi-algebraically
connected component of all realizable
sign conditions of a family of polynomials restricted to the variety.
Because of complexity issues related to algorithmically choosing a generic
system of co-ordinates however,
instead of choosing a single generic system of co-ordinates,
a finite universal family of different co-ordinate systems
was used to approximate the variety. Since in this paper we are
not dealing with algorithmic complexity issues, we are free to choose
generic co-ordinates.
Note also that the idea of bounding the number of semi-algebraically
connected components of realizable sign conditions or of real
algebraic varieties, using known formulas for Betti numbers of non-singular,
complete intersections in complex projective spaces, and then
using Smith inequality,
have been used before
in several different
settings (see \cite{Bas05-first-Kettner} in the case of semi-algebraic
sets defined by quadrics and \cite{Benedetti-Loeser} for arbitrary real
algebraic varieties).

The rest of the paper is organized as follows. In Section \ref{sec:prelim},
we state some known results that we will need to prove the main theorem. These
include explicit recursive
formulas for the sum of Betti numbers of non-singular, complete
intersections of complex projective varieties
(Section \ref{subsec:Betti}), the Smith inequality
relating the Betti numbers of complex varieties defined over $\R$ with those
of their real parts (Section \ref{subsec:Smith}),
some results about generic choice of co-ordinates (Sections \ref{subsec:polar},
\ref{subsec:generic}), and finally a few facts about non-archimedean
extensions and Puiseux series that we need for making perturbations
(Section \ref{subsec:puiseux}).
We prove the main theorem in Section \ref{sec:main}.

\section{Certain Preliminaries}
\label{sec:prelim}
\subsection{The Betti numbers of a non-singular complete intersection
in complex projective space}
\label{subsec:Betti}

If $\mathcal{P}$ is a finite subset of $\R[X_1,\ldots,X_k]$ consisting of homogeneous polynomials we denote
the set of zeros
of ${\mathcal P}$ in $\PP_\R^k$ (resp. in $\PP_\C^k$) by
$$
\displaylines{
\ZZ({\mathcal P},\PP_\R^k)=\{x\in \PP_\R^k\;\mid\;\bigwedge_{P\in{\mathcal P}}P(x)= 0\}
}
$$

$$
\displaylines{
(\mbox{resp. }
\ZZ({\mathcal P},\PP_\C^k)=\{x\in \PP_\C^k\;\mid\;\bigwedge_{P\in{\mathcal P}}P(x)= 0\}
).
}
$$
For $P \in \R[X_1,\ldots,X_k]$ we will denote by $\ZZ(P,\R^k)$ (
resp. $\ZZ(P,\C^k)$, $\ZZ(P,\PP_\R^k)$, $\ZZ(P,\PP_\C^k)$) the variety
$\ZZ(\{P\},\R^k)$
(
resp. $\ZZ(\{P\},\C^k)$, $\ZZ(\{P\},\PP_\R^k)$, $\ZZ(\{P\},\PP_\C^k)$).

For any locally closed semi-algebraic set $X$,
we denote by $b_i(X)$ the dimension
of
\[
\HH_i(X,\Z_2),
\]
the $i$-th homology group of $X$
with coefficients in $\Z_2$. We refer to \cite[Chapter 6]{BPRbook2}
for the definition of homology groups in case the field $\R$ is not
the field of real numbers.
Note that $b_0(X)$ equals the number of semi-algebraically connected
components of the semi-algebraic set $X$.

For $\sigma\in \{0,1,-1\}^{\mathcal P}$ and $V \subset \R^k$ a closed
semi-algebraic set, we will denote by
$b_i(\sigma,V)$ the dimension of
\[
\HH_i(\RR(\sigma,V),\Z_2).
\]

We will denote by
\[
b(\sigma,V) = \sum_{i \geq 0} b_i(\sigma,V).
\]
\begin{definition}
A projective variety $X\subset\PP^k_{\C}$ of codimension~$n$ is a
\textit{non-singular complete intersection}
if it is the intersection of $n$ non-singular
hypersurfaces
in $\mathbb{P}^k_{\C}$
that meet transversally at each point of the
intersection.
\end{definition}

Fix  an $m$-tuple of natural numbers $\bar{d} = (d_1,\ldots,d_m)$. Let
$X_{\C} = \Zer(\{Q_1,\ldots,Q_m\},\mathbb{P}_{\C}^{k})$,
such that the degree of $Q_i$ is $d_i$,
denote a complex projective variety of
codimension~$m$ which is a non-singular complete intersection.
It is a classical fact that the Betti numbers of $X_{\C}$
depend only on the degree sequence and not on the specific $X_{\C}$. 
In fact, it follows from Lefshetz theorem on hyperplane sections 
(see, for example, \cite[Section 1.2.2]{Voisin2})
that 
\[
b_i(X_\C) = b_i(\mathbb{P}_{\C}^{k}), 
\; 0 \leq i < k - m.
\]
Also, by Poincar\'e duality we have that,
\[
b_{i} (X_\C) = 
b_{2(k-m) - i}(X_\C),
\; 0 \leq i \leq k - m.
\]
Thus, all the Betti numbers of $X_\C$
are determined once we know 
$b_{k-m}(X_\C)$ or equivalently
the Euler-Poinca\'re characteristic 
\[
\chi(X_\C)
= \sum_{i \geq 0} 
(-1)^i b_i(X_\C).
\]

Denoting $\chi(X_\C)$ 
by $\chi^k_m(d_1,\dots,d_m)$ (since it only depends on the degree
sequence) we have the following recurrence relation 
(see for example \cite{Benedetti-Loeser}).

\begin{equation}{\label{eqn:chi}}
\chi^k_m(d_1,\dots,d_m)=
\begin{cases}
k+1 &\text{ if } m=0\\
d_1 \ldots d_m &\text{ if } m=k\\
d_m\chi^{k-1}_{m-1}(d_1,\dots,d_{m-1})-(d_m-1)\chi^{k-1}_m(d_1,\dots,d_m)
&\text{ if } 0<m<k
\end{cases}
\end{equation}

We have the following inequality.

\begin{proposition}
\label{prop:ci_bound_NEW}
Suppose $1\leq d_1\leq d_2\leq \dots \leq d_m$.  The function $\chi^k_m(d_1,\ldots,d_m)$ satisfies
$$|\chi^k_m(d_1,\ldots,d_m)|\leq
\textstyle\binom{k+1}{m+1}
d_1\dots d_{m-1} d_m^{k-m+1}.$$
\end{proposition}

\proof
The proof is by induction in each of the three cases of Equation \ref{eqn:chi}.

\case{$m=0$}
$$\chi^k_0=k+1\leq \textstyle\binom{k+1}{1}=k+1  $$

\case{$m=k$}
$$\chi^k_m(d_1,\ldots,d_m)=d_1\ldots d_{m-1}d_m \leq
\textstyle\binom{k+1}{k+1}
d_1\ldots d_{m-1}d_m = d_1 \ldots d_{m-1} d_m$$

\case{$0<m<k$}
$$\begin{aligned}
  |\chi^k_m(d_1,\ldots,d_m)|=&\
  |d_m\chi^{k-1}_{m-1}(d_1,\dots,d_{m-1})-(d_m-1)\chi^{k-1}_{m}(d_1,\ldots,d_m)| \\ \leq
  &\ d_m|\chi^{k-1}_{m-1}(d_1,\dots,d_{m-1})|+d_m|\chi^{k-1}_m(d_1,\dots,d_m)| \\ \leq
  &\ d_m
  \textstyle\binom{k}{m}d_1\ldots d_{m-2}d_{m-1}^{(k-1)-(m-1)+1}+d_m
  \textstyle\binom{k}{m+1}d_1\ldots d_{m-1}d_{m}^{(k-1)-m+1}\\
  \overset{\ast}{\leq}
  &\
  \textstyle\binom{k}{m}d_1\dots d_{m-1}d_{m}^{k-m+1} +
  \textstyle\binom{k}{m+1}d_1\ldots d_{m-1}d_m^{k-m+1} \\ =
  &\ \textstyle\binom{k+1}{m+1} d_1\ldots d_{m-1}d_m^{k-m+1}
\end{aligned}$$
where the inequality $\overset{\ast}{\leq}$ follows from the
observation
$$d_{m-1}^{(k-1)-(m-1)+1} \leq d_{m-1}d_m^{k-m},$$
since $d_{m-1}\leq d_m$ by assumption, and the last equality is from the identity $\binom{k+1}{m+1}=\binom{k}{m}+\binom{k}{m+1}$.

\zz

Now let $\beta^k_m(d_1,\dots,d_m)$ denote 
$\sum_{i \geq 0} b_i(X_\C)$.

The following corollary is an immediate consequence of 
Proposition \ref{prop:ci_bound_NEW} and the remarks preceding it.

\begin{corollary}
\label{cor:bijkm}
\begin{eqnarray*}
\beta^k_m(d_1,\dots,d_m) &\leq&
\textstyle\binom{k+1}{m+1}d_1\ldots d_{m-1}d_m^{k-m+1}+2(k-m+1).
\end{eqnarray*}
\end{corollary}

\subsection{Smith inequality}
\label{subsec:Smith}
We state a version of the
Smith inequality which plays a crucial role in the proof of the main
theorem.
Recall that for any compact topological
space equipped with an involution, inequalities derived from the {\em Smith
exact sequence} allows one to bound the {\em sum} of the Betti numbers
(with $\Z_2$ coefficients)
of the fixed point set of the involution by the sum of the Betti numbers
(again with $\Z_2$ coefficients)
of the space itself (see for instance, \cite{Viro},
p.~131).
In particular, we have for a complex projective
variety defined by real forms, with the involution taken to be
complex conjugation, the following theorem.

\begin{theorem}[Smith inequality]\label{thm:smith}
Let ${\mathcal Q} \subset \R[X_1,\ldots,X_{k+1}]$ be a family of
homogeneous polynomials.
Then,
\[
b(\Zer({\mathcal Q},\PP^k_{\R}))\le b(\Zer({\mathcal Q},\PP^k_{\C})).
\]
\end{theorem}

\begin{remark}
Note that we are going to use Theorem \ref{thm:smith} only for bounding the
number of semi-algebraically connected components (that is
the zero-th Betti number) of certain real varieties. Nevertheless, to
apply the inequality we need a bound on the sum of all the Betti numbers
(not just $b_0$) on the right hand side.
\end{remark}

The following theorem 
used in the proof of Theorem \ref{thm:main} is a direct consequence
of Theorem \ref{thm:smith} and the bound in Corollary
\ref{cor:bijkm}.

\begin{theorem}
\label{thm:main2}
Let $\R$ be a real closed field and
$\mathcal{P} = \{P_1,\ldots,P_m\}
\subset \R[X_1,\ldots,X_k]$ with $\deg(P_i) = d_i, i=1,\ldots,m$,
and
$1\leq
d_1 \leq d_2 \leq \cdots \leq d_m$. 
Let $\mathcal{P}^h = \{P_1^h,\ldots,P_m^h\}$, and
suppose 
that 
$P_1^h,\ldots,P_m^h$ define a non-singular complete intersection in
$\PP_\C^k$. Then,
$$
\displaylines{
b_0(\ZZ(\mathcal{P}^h,\PP_\R^k)) \leq
 \textstyle\binom{k+1}{m+1}
d_1 \cdots d_{m-1} d_m^{k-m+1}+2(k-m+1).
}
$$
In case $\ZZ(\mathcal{P},\R^k)$ is bounded,
$$
\displaylines{
b_0(\ZZ(\mathcal{P},\R^k)) \leq
\textstyle\binom{k+1}{m+1}d_1 \cdots d_{m-1} d_m^{k-m+1}+2(k-m+1).
}
$$
\end{theorem}

\begin{proof}
Proof is immediate from Theorem \ref{thm:smith} and Corollary \ref{cor:bijkm}.
\end{proof}

\begin{remark}
Note that the bound in Theorem \ref{thm:main2} is not true
if we omit the
assumption of being a non-singular complete intersection. A
counter-example is provided by Example \ref{eg:counterexample}.
\end{remark}

\begin{remark}
Another possible approach to the proof of Theorem \ref{thm:main2}
is to use the  critical point method and bound directly the number of 
critical points of a generic projection 
using the multi-homogeneous B\'ezout theorem (see for example
\cite{Safey}).
\end{remark}

\subsection{Generic coordinates}
\label{subsec:generic}
Unless otherwise stated, for any real closed field $\R$, 
we are going to use the Euclidean topology 
(see, for example, \cite[page 26] {BCR}) on $\R^k$. 
Sometimes we will need to use (the coarser) Zariski topology, and 
we explicitly state this whenever it is the case.

\begin{notation}
For a real algebraic set $V=\Zer(Q,\R^k)$
we let $\textnormal{reg } V$ denote the non-singular points
in dimension $\dim V$ of $V$ 
(\cite[Definition 3.3.9]{BCR}).
\end{notation}

\begin{definition}
Let $V=\Zer(Q,\R^k)$ be a real algebraic set.
Define $V^{(k)}=V$, and for $0\leq i \leq k-1$ define
$$V^{(i)}=V^{(i+1)}\setminus \textnormal{reg } V^{(i+1)}.$$
Set $\dim V^{(i)}=d(i)$.
\end{definition}

\begin{remark}{\label{rem:Vi}}
Note that $V^{(i)}$ is
Zariski closed for each $0\leq i \leq k$. 
\end{remark}

\begin{notation}
We denote by $\textnormal{Gr}_\R(k,j)$ the real Grassmannian of $j$-dimensional
linear subspaces of  $\R^k$.
\end{notation}

\begin{notation}
For a real algebraic variety $V \subset \R^k$, and
$x \in \textnormal{reg } V $ where $\dim \textnormal{reg } V = p$, we denote by $T_x V$
the tangent space at $x$ to $V$ (translated to the origin). Note
that $T_x V$ is a $p$-dimensional subspace of $\R^k$, and hence
an element
of $\textnormal{Gr}_\R(k,p)$.
\end{notation}

\begin{definition}{\label{def:good}}
Let $V=\Zer(Q,\R^k)$ be a real algebraic set,
$1\leq j \leq k$, and
$\ell\in \textnormal{Gr}(k,k-j)$.  We say the linear space $\ell$ is
\textit{$j$-good} with respect to $V$ if either:
\blist
\item $j\notin d([0,k])$, or
\item $d(i)=j$,
and
$$
A_\ell:=\{x\in \textnormal{reg } V^{(i)}|\
\dim(\T_x V^{(i)} \cap \ell) =0\}
$$
is a non-empty dense Zariski open subset of $\textnormal{reg } V^{(i)}$.
\elist
\end{definition}

\begin{remark}{\label{rem:Aell}}
Note that the semi-algebraic subset $A_\ell$ is always a
(possibly empty) Zariski open subset of $\textnormal{reg } V^{(i)}$,
hence of $V^{(i)}$.  In the case where $V^{(i)}$ is an
irreducible Zariski closed subset (see Remark \ref{rem:Vi}),
the set $A_\ell$ is either empty or a non-empty dense
Zariski open subset of $\textnormal{reg } V^{(i)}$.
\end{remark}

\begin{definition}
Let $V=\Zer(Q,\R^k)$ be a real algebraic set and
$\mathcal{B}=\{v_1,\dots,v_k\}\subset \R^k$
a basis of $\R^k$.  We say that the basis
$\mathcal{B}$ is \textit{good} with respect to $V$
if for each $j$, $1\leq j \leq k$,
the linear space
$\textnormal{span}\{v_1,\dots,v_{k-j}\}$ is $j$-good.
\end{definition}

\begin{proposition}
{\label{prop:generic}}
Let $V=\Zer(Q,\R^k)$ be a real algebraic set and
$\{v_1,\dots,v_k\}\subset \R^k$ a basis of $\R^k$.
Then, there exists a non-empty open 
semi-algebraic subset
of linear transformations
$\mathcal{O}\subset \textnormal{GL}(k,\R)$ such that for every $T\in \mathcal{O}$ the basis
$\{T(v_1),\dots,T(v_k)\}$ is good with respect to $V$.
\end{proposition}

The proof of Proposition \ref{prop:generic} uses the following
notation and lemma.

\begin{notation}
For any $\ell\in \textnormal{Gr}_\R(k,k-j)$, $1\leq j \leq k$, we denote
by $\Omega(\ell)$
the real algebraic subvariety of $\textnormal{Gr}_\R(k,j)$ defined by
$$\Omega(\ell)=\{\ell'\in \textnormal{Gr}_\R(k,j)| \ \ell \cap \ell'\neq 0\}.$$
\end{notation}

\begin{lemma}
\label{lem:empty}
For any non-empty open 
semi-algebraic subset
$U\subset \textnormal{Gr}_\R(k,k-j)$, $1\leq j \leq k$,
we have
$$\bigcap_{\ell\in U} \Omega (\ell) = \emptyset.$$
\end{lemma}

\begin{proof}
We use a technique due to Chistov et al. (originally appearing in \cite{K})
who explicitly constructed a finite family of elements in
$\textnormal{Gr}_\R(k,k-j)$ such that every $\ell' \in \textnormal{Gr}_\R(k,j)$ is transversal
to at least one member of this family.
More precisely, let
$e_0,\ldots,e_{k-1}$ be the standard basis vectors in $\R^k$, and
let for any $x\in \R$,
$$
\displaylines{
v_k(x) = \sum_{i=0}^{k-1}{x^i} e_i.
}
$$

Then the set of vectors $v_{k}(x),v_k(x+1),\ldots,v_k(x+k-j-1)$ are
linearly independent and span a $(k-j)$-dimensional subspace of $\R^k$.
Denote by $\ell_x$ the corresponding element in $\textnormal{Gr}_\R(k,k-j)$. An easy
adaptation of the proof of Proposition 13.27 \cite{BPRbook2}
now shows that, for $\eps>0$, the set
\[
L_{\eps,k,j} := \{\ell_{m\eps} | 0 \leq m \leq k(k-j) \} \subset \textnormal{Gr}(k,k-j),
\]
has the property that for any $\ell' \in \textnormal{Gr}_\R(k,j)$, there exists
some $m$, $0 \leq m\leq k(k-j)$, such that $\ell' \cap \ell_{m \eps} = 0$.
In other words, for every $\eps > 0$,
$$
\displaylines{
\bigcap_{0 \leq m \leq k(k-j)} \Omega(\ell_{m\eps}) = \emptyset.
}
$$
By rotating co-ordinates we can assume that $\ell_0 \in U$, and
then by choosing $\eps$ small enough we can assume that
$L_{\eps,k,j} \subset U$. This finishes the proof.
\end{proof}

\begin{proof}[Proof of Proposition \ref{prop:generic}]
We prove that for each $j,\; 0 \leq j \leq k$,
the set of $\ell \in \textnormal{Gr}_\R(k,k-j)$ such that $\ell$ is not
$j$-good for $V$ is a semi-algebraic
subset of $\textnormal{Gr}_\R(k,k-j)$ without interior.
It then follows that its complement contains an 
open
dense
semi-algebraic
subset of $\textnormal{Gr}_\R(k,k-j)$, and hence there is an open 
semi-algebraic
subset
$\mathcal{O}_j \subset \textnormal{GL}_n(\R)$ such that for each $T \in \mathcal{O}_j$,
the linear space
$\textnormal{span}\{T(v_1),\dots,T(v_{k-j})\}$ is $j$-good with respect to $V$.

Let $j=d(i)$, $0\leq i \leq k$. Seeking a contradiction,
suppose that there is an open 
semi-algebraic
subset
$U\subset \textnormal{Gr}_\R(k,k-j)$ such that every $\ell\in U$ is not $j$-good
with respect to $V$.
Let
$V^{(i)}_1,\dots,V^{(i)}_n$
be the
distinct
irreducible components of the Zariski closed set $V^{(i)}$.
For each $\ell \in U$,
$\ell$ is not $j$-good for some
$V^{(i)}_r$,
$1\leq r\leq n$ (otherwise $\ell$ would be $j$-good
for $V
$).
Let $U_1,\dots,U_n$ denote
the semi-algebraic sets defined by
$$U_r:=\{\ell\in U| \ \ell \text{ is not } j\text{-good for } V^{(i)}_r\}.$$
We have $U=U_1\cup \dots \cup U_n$, and $U$ is open in $\textnormal{Gr}_\R(k,k-j)$.
Hence,
for some $r$, $1\leq r \leq n$, we have $U_r$
contains an non-empty open 
semi-algebraic
subset.
Replacing $U$ by this (possibly smaller) subset we have that
the set $A_\ell\cap \textnormal{reg } V^{(i)}_r$ is empty
for each $\ell \in U$ (cf. Definition \ref{def:good}, Remark \ref{rem:Aell}).
So,
$\textnormal{reg }V^{(i)}_r\subset \textnormal{reg } V^{(i)}\setminus A_\ell$
for every
$\ell \in U$,
and $$
\emptyset \neq \textnormal{reg } V^{(i)}_r\subset \bigcap_{\ell\in U} \textnormal{reg } V^{(i)}\setminus A_\ell.$$
Let
$\displaystyle{z\in \bigcap_{\ell\in U} \textnormal{reg } V^{(i)}\setminus A_\ell}$, but then
the linear space $\ell' = T_z (\textnormal{reg } V^{(i)})$
is in
$\displaystyle{\bigcap_{\ell \in U} \Omega(\ell)}$,
contradicting Lemma \ref{lem:empty}.
\end{proof}

\subsection{Non-singularity of the set critical of points of
hypersurfaces for generic projections}
\label{subsec:polar}
\begin{notation}
Let $H \in \R[X_1,\ldots,X_k]$. For $0 \leq p \leq k$, we will denote by
$\Cr_p(H)$ the set of polynomials
\[
\{
H, \frac{\partial H}{\partial X_1},\ldots, \frac{\partial H}{\partial X_p}
\}.
\]
We will denote by $\Cr_p^h(H)$ the corresponding set
\[
\{
H^h, \frac{\partial H^h}{\partial X_1},\ldots,
\frac{\partial H^h}{\partial X_p}
\}
\]
of homogenized polynomials.
\end{notation}

\begin{notation}
Let $d$ be even. We will denote by
$\Pos_{\R,d,k} \subset \R[X_1,\ldots,X_k]$
the set of non-negative polynomials in $\R[X_1,\ldots,X_k]$
of degree at most $d$.
Denoting by $\R[X_1,\ldots,X_k]_{\leq d}$ 
the finite dimensional
vector subspace of $\R[X_1,\ldots,X_k]$ consisting of polynomials of degree
at most $d$, we have that $\Pos_{\R,d,k}$ is a (semi-algebraic) cone in
$\R[X_1,\ldots,X_k]_{\leq d}$ with non-empty interior.
\end{notation}

\begin{proposition}
\label{prop:generic_polar}
Let $\R$ be a real closed field and $\C$ the algebraic closure of $\R$.
Let $d>0$ be even. Then there exists $H \in \Pos_{\R,d,k}$, such that
for each $p,\; 0 \leq p \leq k$, $\Cr_p^h(H)$ defines a non-singular complete
intersection in $\PP^k_\C$.
\end{proposition}

\begin{proof}[Proof]
The proposition follows from the fact that the generic polar varieties of
non-singular complex hypersurfaces are non-singular complete intersections
\cite[Proposition 3]{Bank97}, and since $\Pos_{\R,d,k}$
has non-empty interior, we can choose a generic polynomial in
$\Pos_{\R,d,k}$ having this property.
\end{proof}

\begin{remark}

The fact that generic
polar varieties of a non-singular complex variety
are non-singular complete intersections
is not true in general for higher codimension varieties,
see \cite{Bank97,Bank10}, in particular \cite[Section 3]{Bank10}.
\end{remark}

\subsection{Infinitesimals and Puiseux series}
\label{subsec:puiseux}
In our arguments we are going to use infinitesimals and
non-archimedean extensions of a given real closed field $\R$.  A
typical non-archimedean extension of $\R$ is the field $\R\la\eps\ra$
of algebraic Puiseux series with coefficients in $\R$ , which coincide
with the
germs of semi-algebraic continuous functions (see \cite{BPRbook2},
Chapter 2, Section 6 and Chapter 3, Section 3).  An element $x\in \R\la
\eps\ra$ is bounded over $\R$ if $\vert x \vert \le r$ for some $0\le
r \in \R$.  The subring $\R\la\eps\ra_b$ of elements of $\R\la\eps\ra$
bounded over $\R$ consists of the Puiseux series with non-negative
exponents.  We denote by $\lim_{\varepsilon}$ the ring homomorphism
from~$\R \langle \varepsilon \rangle_b$ to $\R$ which maps $\sum_{i
  \in \mathbb{N}} a_i \varepsilon^{i / q}$ to $a_0$. So, the mapping
$\lim_{\varepsilon}$ simply replaces $\varepsilon$ by $0$ in a bounded
Puiseux series.  Given $S\subset \R\la \eps \ra^k$, we denote by
$\lim_\eps(S)\subset \R^k$ the image by $\lim_\eps$ of the elements of
$S$ whose coordinates are bounded over $\R$.  
We denote by $\R\langle \ep_1,\ep_2,\ldots,\ep_\ell\rangle$ the real closed field $\R\langle \ep_1 \rangle \langle \ep_2 \rangle \cdots \langle \ep_\ell \rangle$, and we let $\lim_{\ep_i}$ denote the ring homomorphism $\lim_{\ep_i}\lim_{\ep_{i+1}}\cdots \lim_{\ep_{\ell}}$.  

More generally, let $\R'$ be a real closed field extension of $\R$.
If $S\subset \R^ k$ is a semi-algebraic set, defined by a boolean
formula $\Phi$ with coefficients in $\R$, we denote by $\Ext(S,\R')$
the extension of $S$ to $\R'$, i.e.\ the semi-algebraic subset of
$\R'^k$ defined by $\Phi$.  The first property of $\Ext(S,\R')$ is
that it is well defined, i.e.\ independent on the formula $\Phi$
describing $S$ (\cite{BPRbook2} Proposition 2.87).  Many properties of
$S$ can be transferred to $\Ext(S,\R')$: for example $S$ is non-empty
if and only if $\Ext(S,\R')$ is non-empty, $S$ is semi-algebraically
connected if and only if $\Ext(S,\R')$ is semi-algebraically connected
(\cite{BPRbook2} Proposition 5.24).

\section{Proof of the main theorem}
\label{sec:main}

\begin{remark}
Most of the techniques employed in the proof of the main theorem are similar to those found in \cite{BPRbook2}, see \cite[Section 13.1 and Section 13.3]{BPRbook2}.
\end{remark}

Throughout this section, $\R$ is a real closed field,
$\mathcal{Q},\mathcal{P}$ are finite subsets
of $\R[X_1,\ldots,X_k]$, with
$\deg P= d_P$
for all $P \in \mathcal{P}$, and
$\deg(Q) \leq d_0$ for all $Q \in \mathcal{Q}$.
We
denote by
$k'$ the real dimension of $\ZZ(\mathcal{Q},\R^k)$.
Let $Q = \sum_{F \in \mathcal{Q}} F^2$.

For $x \in \R^k$ and $r > 0$, we will denote by $B_k(0,r)$ the open ball
centred at $x$ of radius $r$. For any semi-algebraic subset $X \subset \R^k$,
we denote by $\overline{X}$ the closure of $X$ in $\R^k$. It follows from
the Tarski-Seidenberg transfer principle (see for example
\cite[Ch 2, Section 5]{BPRbook2}) that the closure
of a semi-algebraic set is again semi-algebraic.

We suppose 
using Proposition \ref{prop:generic} that
after making a
linear change in co-ordinates if necessary 
the given system of co-ordinates
is good
with respect to $\ZZ(Q,\R^k)$.

Using Proposition \ref{prop:generic_polar},
suppose that $H \in \Pos_{\R,2d_0,k}$ satisfies
\begin{property}
\label{prop:H}
for any $p, 0 \leq p \leq k$, $\Cr_p^h(H)$ defines a non-singular complete
intersection in $\PP^k_\C$.
\end{property}

Let $\Def(Q,H,\zeta)$ be defined by
\[
\Def(Q,H,\zeta) = (1 -\zeta)Q - \zeta H.
\]

We first prove several properties of the polynomial $\Def(Q,H,\zeta)$.

\begin{proposition}
\label{prop:limitofdef}
Let 
$\tilde{\R}$ be any real closed field containing $\R\langle 1/\Omega\rangle$, and let 
$C$ be a semi-algebraically connected component of
$$
\ZZ(Q,\tilde{\R}^k)  \cap B_k(0,\Omega).
$$
Then, there exists a semi-algebraically connected component,
$D \subset \tilde{\R}\langle \zeta\rangle^k$  of the semi-algebraic set
$$W = \{ x \in B_k(0,\Omega) \;\mid\; \Def(Q,H,\zeta)(x) \leq 0 \}$$
such that $\overline{C} = \lim_\zeta D$.
\end{proposition}

\begin{proof}
It is clear that $\ZZ(Q,\tilde{\R}\langle \zeta\rangle^k) \cap B_k(0,\Omega)
\subset W$,
since
$H(x) \geq 0$ for all $x\in \tilde{\R}\langle \zeta\rangle^k$. Now let $C$ be
a semi-algebraically connected component of
$\ZZ(Q,\tilde{\R}^k) \cap B_k(0,\Omega)$
and let $D$ be the
semi-algebraically connected component of $W$ containing
$\Ext(C,\tilde{\R}\langle\zeta\rangle)$.
Since $D$ is bounded over $\tilde{\R}$
and semi-algebraically connected
we have $\lim_\zeta D$ is semi-algebraically
connected (using for example Proposition 12.43 in \cite{BPRbook2}),
and contained in $\ZZ(Q,\tilde{\R}^k)
\cap \overline{B_k(0,\Omega)}$.
Moreover, $\lim_\zeta D$ contains $\overline{C}$.
But $\overline{C}$ is a semi-algebraically
connected component of $\ZZ(Q,\tilde{\R}^k) \cap
\overline{B_k(0,\Omega)}$
(using the 
conical structure at infinity of $\ZZ(Q,\tilde{\R}^k)$),
and hence $\lim_\zeta D = \overline{C}$.
\end{proof}

\begin{proposition}
\label{prop:criticallocusontube}
Let
$\tilde{\R}$ be any real closed field containing $\R\langle 1/\Omega\rangle$, and
$$W = \ZZ(\Cr_{k-k'-1}(\Def(Q,H,\zeta)), \tilde{\R}\langle \zeta\rangle^k)
\cap B_k(0,\Omega).
$$
Then, $\lim_\zeta W = \ZZ(Q,\tilde{\R}^k) \cap \overline{B_k(0,\Omega)}$.
\end{proposition}

We will use the following notation.
\begin{notation}
For $ 1\leq p \leq q \leq k$,
we denote by  $\pi_{[p,q]}: \R^k=\R^{[1,k]}\rightarrow \R^{[p,q]}$ the projection
$$(x_1, \ldots, x_k)\mapsto (x_p, \ldots, x_q).$$
\end{notation}

\begin{proof}[Proof of Proposition \ref{prop:criticallocusontube}]
By Proposition \ref{prop:limitofdef} it is clear that
$\lim_\zeta W \subset \ZZ(Q,\tilde{\R}^k) \cap \overline{B_k(0,\Omega)}$.
We prove the other inclusion.

Let $V = \ZZ(Q,\tilde{\R}^k)$,
and suppose that $x \in \textnormal{reg } V^{(i)} \cap B_k(0,\Omega)$ for
some $i, k-k' \leq i \leq k$.
Every open 
semi-algebraic
neighbourhood $U$ of $x$ in
$V \cap B_k(0,\Omega)$
contains a point $y \in \textnormal{reg } V^{(i')} \cap B_k(0,\Omega)$
for some $i' \geq i$, such that
the local dimension of $V$ at
$y$ is equal to $d(i')$.
Moreover, since the given system of co-ordinates is assumed to be good for
$V$,
we can also assume that 
the tangent space
$T_{y}(\textnormal{reg }\; V^{(i')})$
is transverse to the span of the first $k-
d(i')$ co-ordinate
vectors.

It suffices to prove that there exists $z \in W$ such that $\lim_\zeta z = y$.
If this is  true for every neighbourhood $U$ of $x$ in $V$, this
would imply that $x \in \lim_\zeta W$.

Let $p = d(i')$.
The property that
$T_{y}(\textnormal{reg }\; V^{(i')})$
is transverse to the span of the first $k-p$ co-ordinate
vectors implies that $y$ is an isolated point of
$V \cap \pi_{[k-p+1,k]}^{-1}(y)$.
Let $T \subset \tilde{\R}\langle \zeta\rangle^k$
denote the semi-algebraic subset of $B_k(0,\Omega)$
defined by
$$
T = \{ x \in B_k(0,\Omega)\; |\; \Def(Q,H,\zeta)(x) \leq 0 \},
$$
and
$D_y$ denote the semi-algebraically connected component of
$T \cap  \pi_{[k-p+1,k]}^{-1}(y)$ containing $y$. Then,
$D_y$ is a closed and bounded semi-algebraic set, with
$\lim_\zeta D_y = y$.
The boundary of $D_y$ is contained in
\[\ZZ(\Def(Q,H,\zeta),\tilde{\R}\langle \zeta\rangle^k) \cap \pi_{[k-p+1,k]}^{-1}(y).
\]
Let $z \in D_y$ be a point in $D_y$ for which the $(k-p)$-th
co-ordinate achieves
its maximum. Then,
$z \in \ZZ(\Cr_{k-p-1}(\Def(Q,H,\zeta)),\tilde{\R}\langle \zeta\rangle^k)$,
and since $p \leq k'$,
$$\ZZ(\Cr_{k-p-1}(\Def(Q,H,\zeta)),\tilde{\R}\langle \zeta\rangle^k)
\subset
\ZZ(\Cr_{k-k'-1}(\Def(Q,H,\zeta)),\tilde{\R}\langle \zeta\rangle^k)$$
and
hence, $z \in W$. Moreover,
$\lim_\zeta z = y$.
\end{proof}

\begin{proposition}
\label{prop:general_pos_for_cr}
Let $\tilde{\R}$ be any real closed field containing $\R$ and $\tilde{\C}$ the algebraic closure of $\tilde{\R}$.
For every $p,\; 0 \leq p \leq k$, $\Cr_p^h(\Def(Q,H,\zeta))$
defines a non-singular complete intersection in
$\PP^k_{\tilde{\C}\langle \zeta\rangle}$.
\end{proposition}

\begin{proof}
By Property \ref{prop:H} of $H$,
we have that
for each $p,\; 0 \leq p \leq k$, $\Cr_p^h(H)$
defines a non-singular complete intersection in
$\PP^k_{\tilde{\C}}$.
Thus, for each $p,\; 0 \leq p \leq k$, $\Cr_p^h(\Def(Q,H,1))$
defines a non-singular complete intersection in
$\PP^k_{\tilde{\C}}$.
Since the property of being non-singular complete intersection
is first order expressible, the set of
$t \in \tilde{\C}$
for which this holds
is constructible, and since the property is also stable there is an open
subset containing $1$ for which it holds. But since a constructible subset of
$\tilde{\C}$
is either finite or co-finite, there exists
an open interval to the right of $0$ in
$\tilde{\R}$ for which the property holds,
and in particular it holds for infinitesimal $\zeta$.
\end{proof}

\begin{proposition}
\label{prop:nonstrict-to-strict}
Let $\sigma \in \{0,+1,-1\}^{\mathcal{P}}$, and let
let $C$ be a semi-algebraically connected component of
$\RR(\sigma,\ZZ(Q,\R\langle 1/\Omega \rangle^k) \cap B_k(0,\Omega))$.
Then, there exists a unique semi-algebraically connected component,
$D \subset \R\langle 1/\Omega,\eps,\delta\rangle^k$,
of the semi-algebraic set defined by
$$
\displaylines{
(Q = 0)
\wedge
\bigwedge_{P \in \mathcal{P}\atop \sigma(P)=0}
(-\delta  < P < \delta)
\wedge
\bigwedge_{P \in \mathcal{P}\atop \sigma(P)=1} (P  > \eps) \wedge
\bigwedge_{P \in \mathcal{P}\atop \sigma(P)=-1} (P <  -\eps)
\wedge
(|X|^2 < \Omega^2)
}
$$
such that
$C = D \cap \R
\langle 1/\Omega \rangle
^k$.
Moreover, if $C$ is a semi-algebraically connected component of
$\RR(\sigma,\ZZ(Q,\R\langle 1/\Omega \rangle^k)\cap B_k(0,\Omega))$,  
$C'$ is a semi-algebraically connected component of 
$\RR(\sigma',\ZZ(Q,\R\langle 1/\Omega \rangle^k) \cap B_k(0,\Omega))$, and
$D,D'$ are the unique semi-algebraically connected components as above satisfying 
$C=D\cap \R\langle 1/\Omega \rangle^k,\; C'=D'\cap \R\langle 1/\Omega \rangle^k$,
then we have $\overline{D}\cap \overline{D'}=\emptyset$ if $C\neq C'$.

\end{proposition}
\begin{proof}
The first part is clear.  To prove the second part suppose, seeking a contradiction, that $x\in \overline{D}\cap \overline{D'}$.  Notice that $\lim_{\delta} x \in \Ext(C,\R\langle 1/\Omega ,\ep \rangle)\cap \Ext(C',\R\langle 1/\Omega,\ep)$, but $\Ext(C,\R\langle 1/\Omega,\ep \rangle)\cap \Ext(C',\R\langle 1/\Omega,\ep\rangle)=\emptyset$ since $C\cap C'=\emptyset$.
\end{proof}

We denote by $\R'$ the real closed field
$\R\langle 1/\Omega,\eps,\delta\rangle$
and by $\C'$ the algebraic closure of $\R'$.

We also denote
\begin{equation}
\label{eqn:G}
G = \sum_{i=1}^{k} X_i^2 - \Omega^2.
\end{equation}

Let
$\mathcal{P}' \subset \R'[X_1,\ldots,X_k]$
be defined by
$$
\displaylines{
\mathcal{P}' =
\bigcup_{P \in \mathcal{P}} \{P\pm \eps, P\pm \delta\} \cup \{G\}.
}
$$

By Proposition \ref{prop:nonstrict-to-strict} we will henceforth restrict
attention to strict sign conditions on the family $\mathcal{P}'$.

Let 
$\mathcal{H} = 
(H_F \in \Pos_{\R',\deg(F),k}
)_{F \in \mathcal{P}'} $
be a family of polynomials with generic coefficients.  
More precisely, this means that $\mathcal{H}$ is chosen so that it 
avoids a certain Zariski 
closed subset of the product 
$\displaystyle{
\times_{F\in \mathcal{P}'} \Pos_{\R'\langle\zeta\rangle,\deg(F),k}
}
$ 
of codimension at least one, defined by the condition that
$$
\Cr_{k-k'-1}^h(\Def(Q,H,\zeta)) \cup
\bigcup_{F \in \mathcal{P}''}\{H_F^h\}
$$
is not a non-singular, complete intersection in
$\PP_{\C'\langle\zeta\rangle}^k$ for some 
${\mathcal P}'' \subset \mathcal{P}'$.

\begin{proposition}
\label{prop:general_pos}
For each $j, \; 0 \leq  j \leq k'$,
and subset ${\mathcal P}'' \subset \mathcal{P}'$ with
$\card \;\mathcal{P}'' = j$,
and $\tau \in \{-1,+1\}^{\mathcal{P}''}$
the set of homogeneous polynomials
$$
\Cr_{k-k'-1}^h(\Def(Q,H,\zeta)) \cup
\bigcup_{F \in \mathcal{P}''}\{(1 -\eps') F^h  - \tau(F)\; \eps'\; H_F^h\}
$$
defines a non-singular, complete intersection in
$\PP_{\C'\langle\zeta,\eps'\rangle}^k$.
\end{proposition}

\begin{proof}
Consider the family of polynomials,
$$
\displaylines{
\Cr_{k-k'-1}^h(\Def(Q,H,\zeta)) \cup
\bigcup_{F \in \mathcal{P}''}\{(1 - t) F^h - \tau(F)\; t\; H_F^h\}
}
$$
obtained by substituting $t$ for $\eps'$ in the given system.
Since, by Proposition \ref{prop:general_pos_for_cr}
the set $\Cr_{k-k'-1}^h(\Def(Q,H,\zeta))$ defines
a non-singular complete intersection in $\PP_{\C'\langle \zeta \rangle}$, and the $H_F$'s are
chosen generically,
the above system defines a non-singular complete intersection in
$\PP_{\C'\langle \zeta \rangle}$ when $t=1$. The set of $t \in \C'\langle \zeta \rangle$  for which the above system
defines a non-singular, complete intersection is constructible, contains $1$,
and since being a non-singular, complete intersection is a stable condition,
it is co-finite. Hence, it must contain an 
open interval to the right of 0 in $\R'\langle \zeta \rangle$, and hence in particular if we substitute the infinitesimal $\eps'$ for
$t$ we obtain that the system defines a non-singular, complete intersection
in $\PP_{\C'\langle\zeta, \eps'\rangle}^k$.
\end{proof}

\begin{proposition}
\label{prop:main}
Let $\tau \in \{+1,-1\}^{\mathcal{P}'}$ with
$\tau(F) = -1$, if $F = G$
(Eqn. \ref{eqn:G}),
and
let $C$ be a semi-algebraically
connected component of
$\RR(\tau,\ZZ(Q,\R'^k))$.

Then, there exists a
a subset
$\mathcal{P}'' \subset \mathcal{P'}$ with
$\card \;\mathcal{P}''  \leq k'$,
and
a
bounded
semi-algebraically connected component $D$ of
the algebraic set
$$
\displaystyle{
\ZZ(\Cr_{k-k'-1}(\Def(Q,H,\zeta)) \cup
\bigcup_{F \in \mathcal{P}''}\{(1 -\eps') F - \tau(F)\; \eps'\; H_F\},
\R'\langle\zeta,\eps'\rangle^k)
}
$$
such that
$\lim_\zeta D \subset \overline{C}$,
and $\lim_\zeta D \cap C\neq \emptyset$.
\end{proposition}

\begin{proof}
Let $V = \ZZ(Q,\R'^k)$.
By Proposition \ref{prop:criticallocusontube}, with $\tilde{\R}$ substituted by $\R'$, we have that,
for each $x \in C$ there exists
$y \in \ZZ(\Cr_{k-k'-1}(\Def(Q,H,\zeta), \R'\langle \zeta\rangle^k)$ such that
$\lim_\zeta y = x$.

Moreover, using the fact that
$\RR(\tau,\R'^k)$ is open we have that
\[
y \in \ZZ(\Cr_{k-k'-1}(\Def(Q,H,\zeta)), \R'\langle \zeta\rangle^k)
\cap \RR(\tau,\R'^k).
\]

Thus, there exists a semi-algebraically
connected component $C'$ 
of
$$\ZZ(\Cr_{k-k'-1}(\Def(Q,H,\zeta)), \R'\langle\zeta\rangle^k)
\cap \RR(\tau,\R'^k),
$$
such that
$\lim_\zeta C' \subset \overline{C}$,
and $\lim_\zeta C' \cap C \neq \emptyset$.

Note that the closure $\overline{C}$ is a semi-algebraically
connected component of
\[
\RR(\overline{\tau},\ZZ(Q,\R'^k)),
\]
where $\overline{\tau}$ is the formula
$\wedge_{F \in \mathcal{P}'} (\tau(F) F \geq 0)$.

The proof of the proposition now follows the proof of
Proposition 13.2 in \cite{BPRbook2}, and uses
the fact that the set of  polynomials
$\bigcup_{F \in \mathcal{P}'}\{(1 -\eps') F - \tau(F)\eps' H_P\}$
has the property that, no $k'+1$ of distinct elements of this set
can have a common zero in 
$\ZZ(\Cr_{k-k'-1}(\Def(Q,H,\zeta)), \R'\la\zeta,\eps'\ra^k)$
by Proposition \ref{prop:general_pos}.
\end{proof}

\begin{proof}[ Proof of Theorem \ref{thm:main}]
Using the conical structure at infinity of semi-algebraic sets, we have the following equality,
$$ 
\begin{aligned}
&\sum_{\sigma \in \{-1,1,0\}^\mathcal{P}} b_0 (\RR(\sigma,\Zer(\mathcal{Q},\R^k)))\\ 
=&\sum_{\sigma \in \{-1,1,0\}^\mathcal{P}} b_0 (\RR(\sigma,\Zer(\mathcal{Q},\R\langle 1/\Omega \rangle^k)\cap B_k(0,\Omega))).
\end{aligned}
$$

By Proposition \ref{prop:nonstrict-to-strict}
it suffices to bound the number of semi-algebraically
connected components of the realizations
$\RR(\tau, \ZZ(Q,\R'^k)\cap B_k(0,\Omega))$,
where $\tau \in \{-1,1\}^{\mathcal{P}'}$ satisfying
\begin{eqnarray}
\label{eqn:tau}
\tau(F) &=& -1, \mbox{ if } F = G, \nonumber \\
\tau(F) &=&  1   \mbox{ if } F = P - \eps \mbox{ or } F = P + \delta \mbox{ for some } P \in \mathcal{P},\\
\tau(F) &=& -1, \mbox { otherwise}. \nonumber
\end{eqnarray}

Using Proposition \ref{prop:main} it suffices to bound the number
of semi-algebraically connected components which are bounded over
$\R'$ of the real algebraic sets
$$
\displaystyle{
\ZZ(\Cr_{k-k'-1}(\Def(Q,H,\zeta)) \cup
\bigcup_{F \in \mathcal{P}''}\{(1 -\eps') F + \tau(F) \eps' H_F\},
\R'\langle\zeta,\eps'\rangle^k)
}
$$
for all $\mathcal{P}'' \subset \mathcal{P}'$ with
$\card \;\mathcal{P''} =j \leq k'$ and all
$\tau \in \{-1,1\}^{\mathcal{P}''}$
satisfying Eqn. \ref{eqn:tau}.

Using Theorem \ref{thm:main2}
we get that the number of such
components is
bounded by
\[
\textstyle\binom{k+1}{k-k'+\#\mathcal{P}''+1}\; (2 d_0)^{k-k'}\; 
{d}_{\mathcal{P}''}\;\max_{F\in \mathcal{P}''}\{ 2d_0,\deg(F) \}^{k'-\#\mathcal{P}''}+2(k-\#\mathcal{P}''+1),
\]
where ${d}_{\mathcal{P}''} = \prod_{F \in \mathcal{P}''} \deg(F)$.

Each $F\in \mathcal{P}'\setminus \{G\}$ is of the form $F\in \{P\pm \ep,P\pm \delta\}$ for some $P \in \mathcal{P}$, and the algebraic sets defined by each of these four polynomials are disjoint.
Thus, we have that
$$\displaylines{
\sum_{\sigma \in \{0,1,-1\}^{\mathcal{P}}} b_0(\RR(\sigma,\ZZ(Q,\R^k)))
}
$$
is bounded by
$$
\displaylines{
\sum_{{\mathcal{I}\subset \mathcal{P}}\atop \#\mathcal{I}\leq k'}
4^{\#\mathcal{I}}\left(\textstyle\binom{k+1}{k-k'+\#\mathcal{I}+1}  \; (2 d_0)^{k-k'}\; {d}_{\mathcal{I}}\; \max_{P\in \mathcal{I}}\{2d_0,d_P \}^{k'-\#\mathcal{I}}  +2(k-\#\mathcal{I}+1)\right).
}$$
\end{proof}

\section{Applications}
\label{sec:applications}
There are several applications of the bound on the number of semi-algebraically
connected components of sign conditions of a family of real polynomials in
discrete geometry.
We discuss below
an application for bounding the number of geometric
permutations of $n$ well separated convex bodies in $\re^d$ induced
by $k$-transversals.

In \cite{GPW96} the authors reduce the problem
of bounding
the number of geometric
permutations of $n$ well separated convex bodies in $\re^d$ induced
by $k$-transversals
to bounding the number
of semi-algebraically connected components realizable sign conditions
of
\[
{2^{k+1} -2 \choose k}{n \choose k+1}
\]
polynomials in $d^2$ variables,
where each polynomial has degree at most $2k$,
on an algebraic variety (the real Grassmannian of $k$-planes in
$\re^d$) in
$\re^{d^2}$ defined by polynomials of degree $2$. The real Grassmanian
has dimension $k(d-k)$. Applying Theorem \ref{thm:main} we obtain that the
number of semi-algebraically connected components of all realizable
sign conditions in this case is bounded by
$$
\displaylines{
\left (k {2^{k+1} -2 \choose k}{n \choose k+1}\right)^{k(d-k)}\; (O(1))^{d^2},
}
$$
which is a strict improvement of the bound,
$$
\displaylines{
\left ({2^{k+1} -2 \choose k}{n \choose k+1}\right)^{k(d-k)}\; (O(k))^{d^2},
}
$$
in \cite[Theorem 2]{GPW96} (especially in the case
when $k$ is close to $d
)$.

As mentioned in the introduction our bound
might also have some relevance in a new method which has been
developed for  bounding the number of incidences between points and
algebraic varieties of constant degree, using a decomposition
technique
based on the polynomial ham-sandwich cut theorem 
\cite{Guth-Katz,Matousek11b,Solymosi-Tao,Zahl}.

\section{Acknowledgment}
The authors would like to thank J. Matousek for bringing this problem
to their attention as well as providing encouragement and useful advice,
and M.-F. Roy for making helpful comments on a first draft of
the paper.
The authors were partially supported by an NSF grant CCF-0915954.
\bibliographystyle{plain}
\bibliography{master}
\end{document}